\documentclass{amsart}

\usepackage{amsfonts,amssymb,amscd,amsmath,enumerate,verbatim,calc}

\newcommand{\lex}{\operatorname{lex}}

\newcommand{\lm}{\operatorname{LM}}
\newcommand{\lc}{\operatorname{LC}}

\newcommand{\bC}{{\boldsymbol C}}

\newcommand{\BZ}{{\boldsymbol Z}}

\theoremstyle{plain}
\newtheorem{theorem}{Theorem}
\newtheorem{corollary}[theorem]{Corollary}
\newtheorem{lemma}[theorem]{Lemma}

\newtheorem*{ack}{Acknowledgement}

\newtheorem{proposition}[theorem]{Proposition}
\theoremstyle{definition}
\newtheorem{alg}{Algorithm}
\newtheorem*{inp}{Input}
\newtheorem*{out}{Output}
\newtheorem{remark}[theorem]{Remark}
\newtheorem{Example}{Example}

\newtheorem*{alp}{Proof of correctness of Algorithm 1}

\begin{document}

\title[ Monomial Gotzmann sets in a quotient by a pure power  ]
{  Monomial Gotzmann sets in a quotient by a pure power    }
\author{Ata Firat P\.ir}
\author{M\"uf\.it Sezer}

\address { Department of Mathematics, Bilkent University,
 Ankara 06800 Turkey}
\email{pir@fen.bilkent.edu.tr} \email{sezer@fen.bilkent.edu.tr}
\thanks{The authors are  supported by T\"{u}bitak-Tbag/109T384 and
the   second author is also partially supported by
T\"{u}ba-Gebip/2010.}


\begin{abstract}
A homogeneous set of monomials  in a quotient of the polynomial
ring $S:=F[x_1, \dots, x_n]$ is called Gotzmann if the size of
this set grows minimally when multiplied with the variables. We
note that  Gotzmann sets in the quotient $R:=F[x_1, \dots,
x_n]/(x_1^a)$ arise from certain Gotzmann sets in $S$.
Then we partition the monomials in a  Gotzmann set in $S$ with
respect to the multiplicity of $x_i$ and
 show that if the growth of
the  size of a component is
  larger  than the size of a neighboring component, then this  component is a multiple of a
Gotzmann set in $F[x_1, \dots, x_{i-1}, x_{i+1}, \dots  ,x_n]$.
We also  adopt  some properties of the minimal growth of the
Hilbert function in $S$ to $R$.
\end{abstract}

 \maketitle
 \section*{ Introduction}
Let $S=F[x_1, \dots, x_n]$ be a polynomial ring over a field $F$
with $\deg (x_i)=1$ for $1\le i\le n$.   We use the lexicographic
order on $S$ with $x_1> \dots >x_n$. For a homogeneous ideal $I$
in $S$, the Hilbert function $H(I,-):\BZ_{\ge 0}\rightarrow
\BZ_{\ge 0}$ of $I$ is the numerical function defined by
$H(I,t)=\dim_FI_t$, where $I_t$ is the homogeneous component of
degree $t$ of $I$. A set $M$ of monomials in $S$ is called
lexsegment if for monomials $m\in M$ and $v\in S $ we have: If
$\deg m=\deg v$ and $v>m$, then $v\in M$. A monomial ideal $I$ is
called lexsegment if the set of monomials in $I$ is lexsegment.
  For a set of monomials $M$ in the
homogeneous component $S_t$ of degree $t$ in $S$, let $\lex_S (M)$
denote the lexsegment set of $|M|$ monomials in $S_t$.
Also for a set of monomials $M$, $S_1\cdot M$ denotes the set of
monomials of the form $um$, where $u$ is a variable and $m\in M$.
By a classical theorem of Macaulay \cite[C4]{MR1735271} we have
\begin{equation}
\label{ana} |(S_1\cdot \lex_S (M))|\le |(S_1\cdot M)|.
\end{equation}
Since the Hilbert function of a homogeneous ideal is the same as
the Hilbert function of its lead term ideal this inequality
implies that for each homogeneous ideal in $S$ there is a
lexsegment ideal with the same Hilbert function. One course of
research inspired by Macaulay's theorem is the study of the
homogeneous ideals $I$ such that every Hilbert function in $S/I$
is obtained by a lexsegment ideal in $S/I$. Such quotients are
called Macaulay-Lex rings. A well-known  example is $S/(x_1^{a_1},
\dots , x_n^{a_n}) $ with $a_1\le \dots \le a_n\le \infty$ and
$x_i^\infty=0$ which is due to Clements and Lindstr\" {o}m
\cite{MR0246781}. These rings have important applications in
combinatorics and algebraic geometry. For a good account of these
matters and basic properties of Macaulay-Lex rings we direct the
reader to Mermin and Peeva \cite{MR2231127}, \cite{MR2329561}.
Some recently discovered classes    of Macaulay-Lex rings can be
found in Mermin and Murai \cite{MR2300000}.

Monomial sets in $S$ whose sizes grow minimally in the sense of
Macaulay's inequality have also attracted attention: A homogeneous
set $M$ of monomials is called Gotzmann if $|(S_1\cdot \lex_S
(M)|= |(S_1\cdot M)|$ and a monomial ideal $I$ is Gotzmann if the
set of monomials in $I_t$ is a Gotzmann set for all $t$. In
\cite{MR2434473}, Gotzmann ideals in $S$ that are generated by at
most $n$ homogeneous polynomials are classified in terms of their
Hilbert functions.
 In \cite{MR2379725} Murai finds all integers $j$
such that every  Gotzmann set of size $j$ in $S$ is lexsegment up
to a permutation. He also classifies all  Gotzmann sets for $n\le
3$. Gotzmann persistence theorem states that if $M$ is a Gotzmann
set in $S$, then $S_1\cdot M$ is also a Gotzmann set, see
\cite{MR0480478}. In \cite{MR2368639} Murai gives a combinatorial
proof of this theorem using binomial representations. He derives
some properties of these representations which provide information
on the  growth of the Hilbert functions. Among other related
works, Aramova, Herzog and Hibi obtains Macaulay's and Gotzmann's
theorems for exterior algebras, \cite{MR1444495}. More recently,
Hoefel shows that the only edge ideals that are Gotzmann are the
ones that arise from star graphs, see \cite{211000}. Also some
results on generation of lexsegment and Gotzmann ideals by
invariant monomials can be found in \cite{MR11}.

In this paper we study  the Gotzmann sets and the minimal growth
of the Hilbert function in the
 Macaulay-Lex quotient $R:=F[x_1, \dots, x_n]/(x_1^a)$, where $a$
 is a positive integer. A set $M$ of monomials in $R$ can also be considered as a
 set of monomials in $S$ and by $R_1\cdot M$ we mean the set of monomials
 in $S_1\cdot M$ that are not zero in $R$. A set $M$ of
 monomials in $R_t$ is Gotzmann if $|(R_1\cdot \lex_R (M))|= |(R_1\cdot
 M)|$, where $R_t$ is homogeneous component of degree $t$ of $R$ and $\lex_R (M)$ denotes the lexsegment set of monomials
 in $R_t$ that has the same size as $M$.
We show that Gotzmann sets in $R$ arise from certain
Gotzmann sets in S: When a Gotzmann set in $R_t$ with $t\ge a$ is
added to the set of monomials in $S_t$ that are divisible by
$x_1^a$, one gets a Gotzmann set in $S_t$.
Then we partition the monomials in a  Gotzmann set in $S$ with
respect to the multiplicity of $x_i$ and
 show that if the growth of
the  size of a component is
  larger  than the size of a neighboring component, then this  component is a multiple of a
Gotzmann set in $F[x_1, \dots, x_{i-1}, x_{i+1}, \dots  ,x_n]$.
 Otherwise we obtain lower bounds on the size of the
component in terms of sizes of neighboring components.    We also
note own adoptions of some properties concerning the minimal
growth of the Hilbert function in $S$ to $R$.

For a general reference for Hilbert functions and Gotzmann ideals
we recommend \cite{MR1251956} and \cite{MR1648665}.
\section*{  Gotzmann sets   in  $F[x_1, \dots, x_n]/(x_1^a)$ }

We continue with the notation and the convention of the previous
section. For a homogeneous lexsegment set $L$ in $S$ with $|L|=d$,
  the size of $S_1\cdot L$ was computed by Macaulay. This number is very closely
related to the $n$-th binomial representation of $d$ and is
denoted by $d^{<n-1>}$. We refer the reader to \cite[\S
4]{MR1251956} for more information on this number. In contrast to
the situation in $S$, for the homogeneous lexsegment set
$L\subseteq R_t$ of size $d$,  the size of the set $R_1\cdot L$
depends also on $t$. We let $d_{n,t}$ denote this size. In the
sequel when we talk about $d_{n,t}$ we will always assume that $d$
is smaller than the number of monomials in $R_t$  because
otherwise $d_{n,t}$ is not defined.  Notice that we have
$d_{n,t}=d^{<n-1>}$ for $t< a-1$. For a non-negative integer $i$,
let $S_t^i$ and $R_t^i$ denote the set of monomials in $S_t$ and
$R_t$ respectively that are divisible by $x_1^i$
 but not by $x_1^{i+1}$. For a set of monomials $M$ in $R_t$, let
$M^i$ denote the set $R_t^i\cap M$. Similarly, if $M$ is in $S_t$,
then $M^i$ denotes $S_t^i\cap M$. Also let $I(M)$ denote the
smallest integer such that $M^{I(M)}\neq \emptyset$. Set
$S'=F[x_2, \dots, x_n]$ and let $S'_1\cdot M$ denote the set of
monomials of the form $x_im$, where $2\le i\le n$ and $m\in M$.
For a monomial $u\in R$ and a monomial set $M$ in $R$ we let
$u\cdot M$ denote the set of monomials in $R$ that are of the form
$um$ with $m\in M$.  We also let $M^i/x_1$ denote the set of
monomials $m$ in $S'$ such that $mx_1^i\in M^i$.
\begin{lemma}
\label{ilk}Let $L$ be the lexsegment set of size $d$ in $R_t$ with
$t\ge a-1$ and $j$  denote $I(L)$.
 Then  $$d_{n,t}=\sum_{j\le i\le a-1}|L^i|^{<n-2>}.$$
In particular, we have $d_{n,t}\ge d^{<n-2>}$. Moreover,
$d_{n,h}>d^{<n-2>}$ for $0\le h<a-1$.
\end{lemma}
\begin{proof}
Since $d_{n,h}=d^{<n-1>}$ for $h< a-1$, the  final statement is
precisely \cite[1.7]{MR2368639}.

 Since $L$ is
lexsegment, we have $L^i=R_t^i$  for $j< i\le a-1$ giving
$x_1\cdot L^i\subseteq S'_1 \cdot L^{i+1}$ for $j\le i <a-1$.
 Moreover $x_1\cdot L^{a-1}$ is empty and
so we get
\begin{equation*}R_1\cdot L=\bigsqcup_{j\le i\le a-1} S'_1 \cdot
L^i.\end{equation*}
 Note that
 $|L^i/x_1|=|L^i|$ and that  $L^i/x_1$ is a lexsegment set in $S'$.
 Therefore $|S_1'\cdot L^i|=|S_1'\cdot (L^i/x_1)|$ and $|S_1'\cdot
 (L^i/x_1)|=|L^i|^{<n-2>}$. It follows that $d_{n,t}=|R_1\cdot
 L|=\sum_{j\le i\le a-1}|L^i|^{<n-2>}$, as desired. Finally, since $\sum_{j\le i\le
 a-1}|L^i|=d$ we also have $\sum_{j\le i\le
 a-1}|L^i|^{<n-2>}\ge d^{<n-2>}$
   by \cite[1.5]{MR2368639} with equality only if
  $j=a-1$.
\end{proof}
\begin{lemma}
\label{iki} Let $M$ be a set of monomials in $R_t$ with $t\ge a$.
Let $B$ denote the set of monomials in $S_t$ that are divisible by
 $x_1^a$.   We have the disjoint union  $$S_1\cdot (B\sqcup M)=S_1\cdot
   B\sqcup R_1\cdot M.$$
Therefore
 $d_{n,t}=(d+|B|)^{<n-1>}-|B|^{<n-1>}$. In
 particular, $d_{n,t}<d^{<n-1>}$.
\end{lemma}
\begin{proof}
 Since $t\ge a$, $B$ is non-empty. Note also that $B$ is a lexsegment set in $S$ because $x_1$ is
 the highest ranked variable.
  Meanwhile    no monomial in $M$ is divisible by
  $x_1^a$ and hence $M$ and $B$ are disjoint sets. Since $R_1\cdot
  M$ is the set of monomials in $S_1\cdot M$ that are not divisible
  by $x_1^a$, we clearly have $S_1\cdot (B\sqcup M)\supseteq S_1\cdot
   B\sqcup R_1\cdot M$. Conversely, let $m$ be a monomial in $S_1\cdot (B\sqcup
   M)$. We may take $m\in  (S_1\cdot M) \setminus (R_1\cdot M)$.
   Then $m$ is divisible by $x_1^a$ and since the degree of $m$ is
   at least $a+1$, $m/x_1^a$ is divisible by one of the variables,
   say $x_i$. Then $m=x_i(m/x_i)\in S_1\cdot B$.

   If $L$ is the lexsegment set of size $d$ in $R_t$, then we just
   showed that $|S_1\cdot (B\sqcup L)|=|S_1\cdot
   B|+ |R_1\cdot L|$. Moreover, since
  the maximal monomial in $R_t$ that is not in $B$ is the maximal monomial in $L$, we get that
$L \sqcup B$ is a  lexsegment set in $S$. It follows that
$d_{n,t}=|R_1\cdot L|=(d+|B|)^{<n-1>}-|B|^{<n-1>}$, as desired.
The last
 statement now follows from \cite[1.5]{MR2368639}.
\end{proof}

We show that Gotzmann sets in $R_t$ for $t\ge a$ arise from
Gotzmann sets in $S_t$ that contain $B$.

\begin{theorem}
\label{dusurme} Let $M$ be a  set of monomials in $R_t$ for $t\ge
a$. Then $M$ is Gotzmann in $R_t$ if and only if $B\sqcup M$ is
Gotzmann in $S_t$.
\end{theorem}
\begin{proof}
Let $L$ denote the lexsegment set in $R_t$ of the same size as
$M$. Then Lemma \ref{iki} implies that $|R_1\cdot L|=|R_1\cdot M|$
if and only  if $|S_1\cdot (B\sqcup M)|=|S_1\cdot (B\sqcup L)|$.
Hence the statement of the proposition follows because $B\sqcup L$
is lexsegment in $S_t$ as we saw in the proof of Lemma \ref{iki}.
\end{proof}
\begin{remark}
\label{a-1} This theorem does not generalize to all Macaulay-Lex
quotients.
 Consider
the Gotzmann set $A:=\{x_1^3x_2, x_1^3x_3, x_1x_2^3, x_2^3x_3 \}$
in $F[x_1, x_2, x_3]/(x_1^4, x_2^4)$. Then the set $A\cup \{ x_2^4
\}$ is not Gotzmann in $F[x_1, x_2, x_3]/(x_1^4)$. Furthermore
$A\cup \{ x_1^4, x_2^4 \}$ is not Gotzmann in $F[x_1, x_2, x_3]$.

For $t=a-1$,  have $d_{n,a-1}=d^{<n-1>}-1$. Hence if a set of $M$
of monomials does not contain $x_1^{a-1}$, then $R_1\cdot
M=S_1\cdot M$ and so $|R_1\cdot M|>d_{n, a-1}$. On the other hand
if $x_1^{a-1}\in M$, then $|R_1\cdot M|=|S_1\cdot M|-1$. It
follows that $M$ is Gotzmann in $R_{a-1}$ if and only if $M$ is
Gotzmann in $S_{a-1}$ and $x_1^{a-1}\in M$.
\end{remark}
We prove a result on Gotzmann sets in $S$.   Let $M$ be a Gotzmann
set in $S_t$. We show that $M^i$ is a product of $x_1^i$ with a
Gotzmann set in $S'$ if $|M^i|^{<n-2>}$ is  larger than
$|M^{i-1}|$. Otherwise we provide lower bounds on the size of
$M^i$ depending on the sizes of neighboring components. We first
prove the following.
\begin{lemma}
Let $M$ be a Gotzmann set of monomials in $S_t$ with $t\ge 0$.
For $0\le i\le t$ set $d_i=|M^i|$. For $0\le i \le t+1$ we have
$$ |(S_1\cdot M)^i|= \max \{ d_i^{<n-2>}, d_{i-1}\}.$$
\end{lemma}
\begin{proof}
For a set of monomials $K$ in $S_t$ and a monomial $u\in S$, let
$u\cdot K$ denote the set of monomials $uk$, where $k\in K$. Note
that a monomial in $(S_1\cdot K)^i$ is either product of a
variable in $S'$ with a monomial in $K^i$ or a product of $x_1$
with a monomial in $K^{i-1}$. It follows that $(S_1\cdot
K)^i=S_1'\cdot K^i \cup x_1\cdot K^{i-1}$. We also have $S_1'\cdot
K^i=S_1'\cdot (x_1^i\cdot (K^i/x_1))=x_1^i\cdot (S_1' \cdot
(K^i/x_1) )$. Applying this to the set $M$ we get that the size of
the set $S_1'\cdot M^i$ is equal to the size of $S_1' \cdot
(M^i/x_1)$ which is at least $d_i^{<n-2>}$ by Macaulay's theorem.
Meanwhile the size of the set $x_1\cdot M_{i-1}$ is $d_{i-1}$. It
follows for all $0\le i\le t+1$ that
\begin{equation}
\label{kucuk} |(S_1\cdot M)^i|\ge \max \{d_i^{<n-2>}, d_{i-1} \}.
\end{equation}
 Define
\begin{equation*}T=\bigsqcup_{0\le i\le t}x_1^i\cdot
(\lex_{S'}(M^i/x_1)).\end{equation*}
  Notice that we have $|T^i|=d_i$ for $0\le i\le t$. We compute $|(S_1 \cdot T)^i|$ for $0\le i\le t+1$
  as follows.
 From the first paragraph of the proof we have $(S_1\cdot T)^i=S_1'\cdot T^i \cup x_1\cdot
 T^{i-1}$
 and that $S_1'\cdot T^i=x_1^i\cdot (S_1'
\cdot (T^i/x_1) )$. But $T^i/x_1$ is a homogeneous lexsegment set
by construction and so $S_1' \cdot (T^i/x_1)$ is also a lexsegment
set   in $S'_{t-i+1}$, see \cite[4.2.5.]{MR1251956}. Hence
$|S_1'\cdot T^i|=d_i^{<n-2>}$. On the other hand $|x_1\cdot
T^{i-1}|=d_{i-1}$. Moreover, since $T^{i-1}/x_1$ is a lexsegment
set in $S'_{t-i+1}$, the identity $x_1\cdot T^{i-1}=x_1^i\cdot
(T^{i-1}/x_1)$ gives that $x_1\cdot T^{i-1}$ is obtained by
multiplying each element in a homogeneous lexsegment set in $S'$
with $x_1^i$. Since $S_1'\cdot T^i$ is also obtained by
multiplying the lexsegment set $S_1' \cdot (T^i/x_1)$ with $x_1^i$
we have either $S_1'\cdot T^i\subseteq x_1\cdot T^{i-1}$ or
$S_1'\cdot T^i\supseteq x_1\cdot T^{i-1}$. Hence $(S_1\cdot
T)^i=S_1'\cdot T^i $ if $d_i^{<n-2>}\ge d_{i-1}$ and $(S_1\cdot
T)^i=x_1\cdot
 T^{i-1}$ otherwise. Moreover, $|(S_1\cdot T)^i|=\max \{d_i^{<n-2>}, d_{i-1}
 \}$. Since the size of $M$ has the minimal possible growth,
 from Inequality \ref{kucuk} we get $|(S_1\cdot M)^i|=\max \{d_i^{<n-2>}, d_{i-1}
 \}$ as desired.
\end{proof}

 We remark
 that the statement of the following theorem    stays true if we permute the variables and
write  the assertion with respect to another variable. It is also
instructive to compare this with \cite[2.1]{MR2379725}.

\begin{theorem}
\label{gotz}  Assume the notation of the previous lemma.  If
$d_i^{<n-2>}\ge d_{i-1}$, then $M^i/x_1$ is Gotzmann in $S'$.
  Moreover, if
$d_i^{<n-2>}< d_{i-1}$, then we have either
$(d_i+1)^{<n-2>}>d_{i-1}-1$ or $d_i+1>d_{i+1}^{<n-2>}$.
\end{theorem}
\begin{proof}
 Assume that $d_i^{<n-2>}\ge d_{i-1}$ for some $0\le i\le t+1$. Then from the first statement we have
  $|(S_1\cdot M)^i|=
 d_i^{<n-2>}$. But $S_1'\cdot M^i$ is a subset of $(S_1\cdot M)^i$ and
 $|S_1'\cdot M^i|=|x_1^i\cdot (S_1'\cdot (M^i/x_1))|=|S_1'\cdot (M^i/x_1)|\ge d_i^{<n-2>}$.
 It follows that $|S_1'\cdot
 (M^i/x_1)|=d_i^{<n-2>}$and so $M^i/x_1$ is Gotzmann.

 We now prove the second assertion of the theorem. Assume that
 there exists an integer $1\le q\le t$ such that $d_q^{<n-2>}<
 d_{q-1}$. By way of contradiction assume further that $(d_q+1)^{<n-2>}\le d_{q-1}-1$ and $d_q+1\le
 d_{q+1}^{<n-2>}$.
 We obtain a contradiction by constructing a  set $W$ in $S_t$
 whose size grows strictly less than the size of $M$. Let
 $w_{q-1}$ be the minimal monomial in $T^{q-1}$.  Notice also that $d_q^{<n-2>}<
 d_{q-1}$ implies that $S_t^q\setminus
  T^q\neq \emptyset$ and let $w_q$ be the
  monomial that is maximal among the monomials in $S_t^q\setminus
  T^q$. Define $$W=\big (\bigsqcup_{0\le i\le t, \; i\neq q-1, q}T^i \big ) \sqcup (T^{q-1}\setminus \{w_{q-1}\}) \sqcup ( T^{q}\cup
  \{w_q\} ).$$
    Notice that by construction
  $W^i/x_1$ is a lexsegment set in $S'$ for all $0\le i\le t$.
Therefore, just as we saw for $T$, we have $|(S_1\cdot W)^i|=\max
\{|W^i|^{<n-2>},
  |W^{i-1}| \}$.   We also have $|W^i|=d_i$ for $i\neq q-1, q$, and  $|W^{q-1}|=d_{q-1}-1$ and $|W^q|=d_q+1$.
  It follows that $|(S_1\cdot T)^i|=|(S_1\cdot W)^i|$ for all $i\neq
  q-1, q, q+1$. We finish the proof by showing that
  $$\sum_{q-1\le i\le q+1}|(S_1\cdot W)^i|<\sum_{q-1\le i\le q+1}|(S_1\cdot
  T)^i|.$$
  We have  $|(S_1\cdot W)^{q-1}|=\max
\{(d_{q-1}-1)^{<n-2>},
  d_{q-2} \}\le \max
\{(d_{q-1})^{<n-2>},
  d_{q-2} \}= |(S_1\cdot T)^{q-1}|.$ Notice also that $|(S_1\cdot W)^q|=\max
\{(d_q+1)^{<n-2>},
  d_{q-1}-1 \}=d_{q-1}-1<d_{q-1}= \max
\{d_q^{<n-2>},
  d_{q-1} \}=|(S_1\cdot T)^q|.$ Finally, $|(S_1\cdot W)^{q+1}|=\max
\{d_{q+1}^{<n-2>},
  d_q+1 \}=d_{q+1}^{<n-2>}=|(S_1\cdot T)^{q+1}|.$
\end{proof}
 We generalize some properties of the minimal growth of the Hilbert
 function in $S$ to $R$. Firstly, we show that $d_{n,t}$ is
increasing in the first parameter and decreasing in the second
parameter.

\begin{proposition}
\label{prop1} Let $d$, $n$, $t$ be positive integers. Then the
following statements hold:
\begin{enumerate}
\item $d_{n+1,t}>d_{n,t}$.

\item $d_{n,t+1}\le d_{n,t}$.
\end{enumerate}
Moreover, for $n\ge 3$ we have $d_{n,t'}=d^{<n-2>}$ for $t'$
sufficiently large.
\end{proposition}
\begin{proof}
   Since $d_{n,t}=d^{<n-1>}$ for
 $t<a-1$, the first statement is precisely \cite[1.7]{MR2368639}
 for $t<a-1$. Let $L$ be the lexsegment set of size $d$ in $R_t$.
 For $t=a-1$, we have $R_1\cdot L=(S_1\cdot L)\setminus \{x_1^a\}$.
 So $d_{n,a-1}=d^{<n-1>}-1$ and the first statement again follows
 from \cite[1.7]{MR2368639}. For $t\ge a$ by
Lemma \ref{iki} we have $d_{n,t}<d^{<n-1>}$. On the other hand,
$d_{n+1,t}\ge
 d^{<n-1>}$ by Lemma \ref{ilk}. This establishes  the first statement.

 Since $d_{n,t}=d^{<n-1>}$ for $t\le a-2$, the second statement holds trivially for
 $t<a-2$. Moreover, we have $d_{n,a-1}=d^{<n-1>}-1$  from the previous paragraph and so
 $d_{n,a-1}<d_{n,a-2}$ as well. Also
 we
 eliminate the case $n=2$ because $d_{2,t}=d$ for $t\ge a-1$.
 So we assume that $t\ge a-1$ and $n>2$. Note that $n>2$
  implies that $|R_t^i|<|R_{t+1}^i|$ and $|R_t^i|<|R_t^{i-1}|$
 for $t\ge a-1$ and $i\le a-1$.  Let $L_1$ and $L_2$ be two lexsegment sets of equal sizes in $R_t$ and $R_{t+1}$
 respectively. The rest of the proof of the second statement is devoted to showing $|R_1\cdot L_1|\ge |R_1\cdot
 L_2|$.

 Set $j_1=I(L_1)$ and $j_2=I(L_2)$. Since $L_1$ and $L_2$ are  lexsegment sets,
  we have $L_1^i=R_t^i$   for $j_1<i\le
a-1$ and $L_2^i=R_{t+1}^i$ for $j_2<i\le a-1$. We also have
$|L_1^{j_1}|\le |R_t^{j_1}|$ and $|L_2^{j_2}|\le |R_{t+1}^{j_2}|$.
  But since $|R_t^i|< |R_{t+1}^i|$, we have
$|L_1^i|<|L_2^i|$ for all $i$ that is strictly bigger than both
$j_1$ and $j_2$. Therefore $j_1\le j_2$ because the sizes of $L_1$
and $L_2$ are the same. Also note that $|R_t^i|=|R_{t+1}^{i+1}|$
and so $|L_1^{i}|=|L_2^{i+1}|$ for $\max\{j_1, j_2-1\}<i< a-1$. We
claim that $j_1+1\ge j_2$. Otherwise we obtain a contradiction as
follows. We have $\max\{j_1, j_2-1\}=j_2-1\neq j_1$ and
$|L_1^{j_2-1}|=|R_t^{j_2-1}|=|R_{t+1}^{j_2}|\ge |L_2^{j_2}|$.
Therefore
$$|L_2|=\sum_{j_2\le i\le a-1}|L_2^i|\le
 \sum_{j_2-1\le i\le a-2}|L_1^i|<\sum_{j_2-1\le i\le a-1}|L_1^i|<\sum_{j_1\le i\le a-1}|L_1^i|=
|L_1|.$$ Thus we  have either $j_1=j_2$ or $j_1+1=j_2$. We handle
these cases separately.

We first assume that $j_1=j_2$. Set $j=j_1$. If $j=a-1$, then by
Lemma \ref{ilk} we have $|R_1\cdot
L_1|=|L_1^{a-1}|^{<n-2>}=|L_1|^{<n-2>}$ and similarly we get
$|R_1\cdot L_2|=|L_1|^{<n-2>}$ giving  the desired inequality. So
assume that $ j<a-1$. Since $|L_1|=|L_2|$ and
 $|L_1^i|=|L_2^{i+1}|$ for $j+1<i< a-1$,
we have  $|L_1^{a-1}|+|L_1^j|=|L_2^j|+|L_2^{j+1}|$. Moreover,  by
Lemma \ref{ilk}, we have $$|R_1\cdot L_1|-|R_1\cdot
L_2|=|L_1^{a-1}|^{<n-2>}+|L_1^j|^{<n-2>}-|L_2^j|^{<n-2>}-|L_2^{j+1}|^{<n-2>}.$$
If $|L_1^j|=|L_2^{j+1}|$, then $|L_1^{a-1}|=|L_2^j|$ as well and
the right hand side of the equation above is zero giving
$|R_1\cdot L_1|=|R_1\cdot
 L_2|$. If
$|L_1^j|\neq |L_2^{j+1}|$, then we necessarily have $|L_1^j|<
|L_2^{j+1}|$ because $L_1^j$ is a subset of $R_t^j$ whose size is
equal to the size of $R_{t+1}^{j+1}=L_2^{j+1}$. Moreover
$|L_1^{a-1}|=|R_t^{a-1}|<|R_{t+1}^{a-1}|\le |R_{t+1}^{j+1}|$
because $j+1\le a-1$ and so $|L_1^{a-1}|<|L_2^{j+1}|$ as well. Now
we get that $|L_2^j|<|L_2^{j+1}|$ because both $|L_1^j|$ and
$|L_1^{a-1}|$ is strictly smaller than $|L_2^{j+1}|$ and
$|L_1^{a-1}|+|L_1^j|=|L_2^j|+|L_2^{j+1}|$. But $|L_2^{j+1} |$ is
the number of all monomials of degree $t-j$ in $n-1$ variables.
Hence \cite[1.6]{MR2368639} applies and  we get
$$|L_1^{a-1}|^{<n-2>}+|L_1^j|^{<n-2>}-|L_2^j|^{<n-2>}-|L_2^{j+1}|^{<n-2>}\ge 0.$$
 Equivalently, $|R_1\cdot L_1|\ge |R_1\cdot
 L_2|$ as desired.

If $j_1+1=j_2$, then from   $|L_1^i|=|L_2^{i+1}|$ for $ j_1<i<a-1$
and Lemma \ref{ilk} we have
$|L_1^{a-1}|+|L_1^{j_1}|=|L_2^{j_1+1}|$ and $$|R_1\cdot
L_1|-|R_1\cdot
L_2|=|L_1^{a-1}|^{<n-2>}+|L_1^{j_1}|^{<n-2>}-|L_2^{j_1+1}|^{<n-2>}.$$
So we get that $|R_1\cdot L_1|-|R_1\cdot L_2|>0$ from
\cite[1.5]{MR2368639}.

Finally, fix an integer $d$.  Then  for a sufficiently large
integer $t'$, the number of all monomials of degree $t'-a-1$ in
$S'$ is bigger than $d$ ($n>2$ is essential here). Now let $L$ be
the lexsegment set of size $d$ in $R_{t'}$. Then we have
$L^{a-1}=L$ and therefore  Lemma \ref{ilk} gives
$d_{n,t'}=d^{<n-2>}$ as desired.
\end{proof}

 Let $L_1$ and $L_2$ be two lexsegment sets in $R_t$ with sizes
$b$ and $c$ respectively with $b\ge c$. Let   $j_1$ and $j_2$
denote $I(L_1)$ and $I(L_2)$. Note that we have $j_1\le j_2$.
Define
$$t'=\begin{cases}
t-j_2 \text{ if } j_1=j_2 \; \text {and}\; j_1\neq a-1 \;
\text{and} \;
x_1^{j_1}x_n^{t-j_1}\notin (L_1\cup L_2); \\
t+1-j_2 \text{ otherwise. }
\end{cases}
 $$

  We finish by noting down an adoption of
\cite[1.5]{MR2368639} for the ring $R$.

\begin{proposition}
Assume the notation of the previous paragraph. For $n\ge 3$ we
have
$$b_{n,t}+c_{n,t}>(b+c)_{n,t+t'}.$$
\end{proposition}
\begin{proof}
Note that for $t<a-1$, the statement of the lemma follows from
\cite[1.5]{MR2368639} because then $b_{n,t}=b^{<n-1>}$,
$c_{n,t}=c^{<n-1>}$ and $(b+c)_{n,t+t'}\le (b+c)^{<n-1>}$. So we
assume $t\ge a-1$.

To prove the lemma it is enough to show that there exist monomials
$m_1$ and $m_2$ of degree $t'$ in $S'$ such that $m_1\cdot L_1$
and $m_2\cdot L_2$ are disjoint and that $R_1\cdot (m_1\cdot L_1)$
and $R_1\cdot (m_2\cdot L_2)$ have non-empty intersection. Because
then
$$(b+c)_{n,t+t'}\le |R_1 \cdot \big (m_1\cdot L_1 \sqcup m_2\cdot L_2 \big )|<
|R_1\cdot  (m_1\cdot L_1)|+|R_1\cdot  (m_2\cdot
L_2)|=b_{n,t}+c_{n,t}.$$ In the following we use the fact  that if
a minimal element in a set is of higher rank than the maximal
element in another set then these two sets do not intersect.

 We handle the case
$j_1=j_2=a-1$ separately and the proof for this case essentially
carries over from \cite[1.5]{MR2368639}.   Let $x_1^{a-1}w$ be the
minimal element in $L_1$, where $w$ is a monomial in $S'$.
Consider $x_2^{t'} \cdot L_1$ and $wx_n \cdot L_2$. The minimal
element of $x_2^{t'} \cdot L_1$ is $x_1^{a-1}x_2^{t'}w$ and the
maximal element of  $wx_n \cdot L_2$ is
$x_1^{a-1}x_2^{t-a+1}wx_n$. Then $x_1^{a-1}x_2^{t'}w>
x_1^{a-1}x_2^{t-a+1}wx_n$ since $t'=t-a+2$. On the other hand
$x_1^{a-1}x_2^{t'}wx_n$ lie in both $R_1\cdot (x_2^{t'}\cdot L_1)$
and $R_1\cdot (wx_n\cdot L_2)$.

It turns out that for all the remaining cases one can choose
$m_1=x_2^{t'}$ and $m_2=x_n^{t'}$. We first show $x_2^{t'}\cdot
L_1 \cap x_n^{t'}\cdot L_2=\emptyset$ and to this end it suffices
to show that $x_2^{t'}\cdot L_1^i \cap x_n^{t'}\cdot
L_2^i=\emptyset$ for $i\le a-1$.

 Since $L_1^i=R_t^i$ for $i>j_1$, the minimal element in
$x_2^{t'}\cdot L_1^i$ is $x_1^ix_2^{t'}x_n^{t-i}$  for $i>j_1$.
Similarly, $L_2^i=R_t^i$ for $i>j_2$, so the maximal element in
$x_n^{t'}\cdot L_2^i$  is $x_1^ix_2^{t-i}x_n^{t'}$. But $t'\ge
t-i$ for $i>j_2$ and so we have $x_1^ix_2^{t'}x_n^{t-i} >
x_1^ix_2^{t-i}x_n^{t'}$
 for
$j_2<i$. Also $L_2^i=\emptyset$ for $i<j_2$ and therefore it only
remains to show $x_2^{t'}\cdot L_1^{j_2}\cap x_n^{t'}\cdot
L_2^{j_2}=\emptyset $. Let $u$ and $v$ denote the maximal monomial
in  $x_n^{t'}\cdot L_2^{j_2}$ and minimal monomial in
 $x_2^{t'}\cdot L_1^{j_2}$, respectively. Note that $u=x_1^{j_2}x_2^{t-j_2}x_n^{t'}$.
We consider different cases and show that $v>u$ in each case.
 First assume that  $j_2>j_1$. Then we have $L_1^{j_2}=R_t^{j_2}$ and
 $t'=t+1-j_2$. It follows that
 $v=x_1^{j_2}x_2^{t'}x_n^{t-j_2}=x_1^{j_2}x_2^{t+1-j_2}x_n^{t-j_2}>u$.
 Now assume that $j_1=j_2$ and $x_1^{j_2}x_n^{t-j_2}\in (L_1\cup
L_2)$. Then $t'=t+1-j_2$. Also, since $|L_1|\ge |L_2|$ and
$x_1^{j_2}x_n^{t-j_2}$ is the minimal element in $R_t^{j_2}$, we
have $x_1^{j_2}x_n^{t-j_2}\in L_1^{j_2}$ and so
$v=x_1^{j_2}x_2^{t'}x_n^{t-j_2}=x_1^{j_2}x_2^{t+1-j_2}x_n^{t-j_2}>u$.
Finally, if $j_1=j_2$ and $x_1^{j_2}x_n^{t-j_2}\notin (L_1\cup
L_2)$, then the minimal monomial in $L_1^{j_2}$ is bigger than
$x_1^{j_2}x_n^{t-j_2}$ and so
$v>x_1^{j_2}x_2^{t'}x_n^{t-j_2}=x_1^{j_2}x_2^{t-j_2}x_n^{t-j_2}=u$.

We now show that $R_1\cdot (x_2^{t'}\cdot L_1)$ and $R_1\cdot
(x_n^{t'}\cdot L_2)$ have non-empty intersection. If $j_1<j_2$,
then as we saw above $L_1^{j_2}=R_t^{j_2}$ and so
$x_1^{j_2}x_n^{t-j_2}$ is the minimal element in $L_1^{j_2}$ and
hence $x_1^{j_2}x_2^{t+1-j_2}x_n^{t-j_2+1}\in R_1\cdot
(x_2^{t'}\cdot L_1)$. But since $x_1^{j_2}x_2^{t-j_2}$ is the
maximal element in $L_2^{j_2}$,
$x_1^{j_2}x_2^{t-j_2+1}x_n^{t+1-j_2}\in R_1\cdot (x_n^{t'}\cdot
L_2)$ as well. For the case $j_1=j_2$ first assume that
$x_1^{j_2}x_n^{t-j_2}\in (L_1\cup L_2)$. We have computed the
maximal monomial $u$ in $x_n^{t'}\cdot L_2^{j_2}$ and the minimal
monomial $v$ in
 $x_2^{t'}\cdot L_1^{j_2}$ for this case in the previous paragraph. Notice that
 we have $vx_n=ux_2$ giving $R_1\cdot (x_2^{t'}\cdot L_1)
\cap R_1\cdot (x_n^{t'}\cdot L_2) \neq \emptyset$. Finally, assume
$j_1=j_2< a-1$ and $x_1^{j_1}x_n^{t-j_1}\notin (L_1\cup L_2)$.
Then $L_1^{j_2+1}=L_2^{j_2+1}=R_t^{j_2+1}$ and so the minimal
element in $x_2^{t'}\cdot L_1^{j_2+1}$ is
$x_1^{j_2+1}x_2^{t-j_2}x_n^{t-1-j_2}$ giving
$x_1^{j_2+1}x_2^{t-j_2}x_n^{t-j_2}\in R_1\cdot (x_2^{t'}\cdot
L_1)$. But this monomial is in $R_1\cdot (x_n^{t'}\cdot L_2)$ as
well because the maximal element in $x_n^{t'}\cdot L_2^{j_2+1}$ is
$x_1^{j_2+1}x_2^{t-1-j_2}x_n^{t-j_2}$.
\end{proof}

\begin{ack} We thank Satoshi Murai for pointing out the assertion
of Theorem \ref{dusurme} to us.
\end{ack}

\bibliographystyle{plain}
\bibliography{Bibliography_Version_10}
\end{document}